\documentclass[preprint,12pt]{elsarticle}

\usepackage[utf8]{inputenc} 
\usepackage[T1]{fontenc}    
\usepackage{hyperref}       
\usepackage{url}            
\usepackage{booktabs}       
\usepackage{amsfonts}       
\usepackage{nicefrac}       
\usepackage{microtype}      
\usepackage{xcolor}         
\usepackage{amsmath}
\usepackage{amsthm}
\usepackage{amssymb}
\usepackage{algorithm2e}
\usepackage{float}
\usepackage{multirow}
\usepackage[inline]{enumitem}

\newtheorem{theorem}{Theorem}
\newtheorem{lemma}{Lemma}

\newtheorem{cor}{Corollary}

\makeatletter
\newcommand*{\inlineequation}[2][]{%
  \begingroup
    \refstepcounter{equation}%
    \ifx\\#1\\
    \else
      \label{#1}
    \fi
    \relpenalty=10000
    \binoppenalty=10000
    \ensuremath{
      #2
    }
    ~\@eqnnum
  \endgroup
}
\makeatother

\usepackage{xcolor}
\usepackage{amsmath}
\usepackage{amssymb}
\usepackage{mathtools}

\usepackage{graphicx}
\usepackage{subfigure} 
\usepackage{caption}

\newcommand*\diff{\mathop{}\!\mathrm{d}}

\begin{document}

\begin{frontmatter}

\title{Decentralised convex optimisation with probability-proportional-to-size quantization}

\author{D.A.~Pasechnyuk}
\affiliation{
organization={Mohamed bin Zayed University of Artificial Intelligence},
country={UAE}}
\affiliation{
organization={Ivannikov Institute for System Programming of the Russian Academy of Sciences},
country={Russia}}

\author{P. Dvurechensky}
\affiliation{
organization={Weierstrass Institute for Applied Analysis and Stochastics},
country={Germany}}

\author{C.A.~Uribe}
\affiliation{
organization={Rice University},
country={USA}}

\author{A.V.~Gasnikov}
\affiliation{
organization={Moscow Institute of Physics and Technology},
country={Russia}}


\begin{abstract}
Communication is one of the bottlenecks of distributed optimisation and learning. To overcome this bottleneck, we propose a novel quantization method that transforms a vector into a sample of components' indices drawn from a categorical distribution with probabilities proportional to values at those components. Then, we propose a primal and a primal-dual accelerated stochastic gradient methods that use our proposed quantization, and derive their convergence rates in terms of probabilities of large deviations. We focus on affine-constrained convex optimisation and its application to decentralised distributed optimisation problems. To illustrate the work of our algorithm, we apply it to the decentralised computation of semi-discrete entropy regularized Wasserstein barycenters.
\end{abstract}


\end{frontmatter}

\section{Introduction}
Recent machine learning applications require to process of a large amount of information in order to train a sufficiently good model. Moreover, such models have a huge number of parameters and their training requires solving large-scale optimization problems. A remedy to this scaling issue is the use of distributed/decentralised optimization and machine learning, i.e., the approach that distributes the computational load between a network of computational devices or agents. This approach also allows respecting privacy constraints and covers the setting when the data is not only stored distributedly but is also produced distributedly, therefore infeasible to accumulate in one machine. Despite the popularity  and effectiveness of this approach, one of its main bottlenecks is sending information between the devices, as decentralisation implies some form of communication between the devices: sending large amounts of information leads to delays, interruptions, and losses. In this paper, we focus on the task of reducing the amount of information sent between agents without sacrificing the quality of the approximate solution of the optimization problem.

\textbf{Contributions.} In this paper, we propose a quantization method for distributed stochastic (or deterministic) optimisation that exploits the magnitude of the entries of the vector to be communicated via random sampling; we name this method Probability-Proportional-to-Size (PPS). Our procedure splits the components of this vector into two parts: non-negative and negative components. Each part is then mapped to a vector from probability simplex, which is used to define a probability distribution over the components of the vector. Then, we use this probability distribution to sample a zero-one vector as a quantized version of the initial vector. This results in an unbiased stochastic estimator and communication-efficient bit vector for distributed optimisation. 
In terms of the number of communication bits, our operator matches the best existing nearly optimal compression operators.
At the same time, when the initial vector belongs to a simplex, our quantization operator leads to a second-order moment smaller than other compression operators available in the literature. This particular case is important since a number of large-scale learning problems, such as logistic regression and regularized optimal transport objectives, have this structural property, namely, the gradient of the objective always belongs to a simplex. 

Leveraging our quantization operator, we propose an accelerated quantized stochastic gradient method for problems with linear constraints. Moreover, we propose an accelerated distributed algorithm with quantization for distributed stochastic optimisation problems. To the best of our knowledge, this is the first accelerated distributed algorithm that combines stochasticity, quantization, and primal-dual properties. The latter means that the main iterates are made in the dual space, but the convergence in the primal variable also has an accelerated rate. In contrast to the existing literature, we provide a convergence rate analysis in terms of the probability of large deviation for the objective residual and consensus constraints. Note that existing analyses are based on convergence in expectation, and a particular run of an algorithm may have large objective residuals and constraint infeasibility. As an intermediate result of independent interest, we also propose an accelerated algorithm with quantization for unconstrained stochastic optimization. 

\textbf{Motivating application.} We exploit the structure of the proposed quantization method to develop a new algorithm for the decentralised computation of approximate semi-discrete Wasserstein barycenters. Optimal transport (OT) has become essential in modern machine learning (ML) due to its ability to consider data geometry in computations. Its applications span image retrieval \citep{rubner2000earth}, image classification \citep{cuturi2013sinkhorn}, and Generative Adversarial Networks \citep{arjovsky2017wasserstein}. One significant use of Wasserstein distances is defining the Fr\'{e}chet mean of distributions, known as the Wasserstein barycentre (WB) \citep{agueh2011barycenters,Boissard2015,ebert2017construction,Panaretos2020FrechetMeans}. WB can ``average'' images interpreted as discrete probability distributions, enabling applications like image morphing, atmospheric data averaging, graph learning, and fairness in ML \cite{barre2020averaging,chzhen2020fair,simon2020barycenters,simou2020node2coords}. However, the increasing scale of ML applications also raises the computational cost of using WB. For instance, calculating the WB for two images with one million pixels involves solving a large-scale optimization problem with approximately $10^{12}$ variables. Achieving high accuracy in WB approximations necessitates processing many samples, leading to polynomial dependencies on distribution support size, number of distributions, and desired accuracy~\cite{borgwardt2019computational}. Researchers have proposed various methods, including heuristics \cite{bouchet2020primal}, strongly-polynomial 2-approximation \cite{borgwardt2020lp}, fast computation techniques \cite{guminov2019acceleratedAM,guo2020fast}, saddle-point methods \cite{tiapkin2020stochastic}, and proximal methods \cite{stonyakin2019gradient, xie2020fast}. Additionally, multi-marginal optimal transport approaches are being explored \cite{2019arXiv191000152L, tupitsa2020multimarginal}.


\textbf{Related works.} 
\textit{Decentralised optimisation.} We limit ourselves to a review of works where the synchronous primal-dual algorithms were proposed and analysed under the smoothness assumption without assuming strong convexity. Primal-dual nature of the algorithm appears to be convenient for the main application considered in this paper, which also does not satisfy strong convexity in its original formulation. At the same time, synchronicity imposes some restrictions on only the communication protocol, which remains implementable in practice. For a broader overview of alternative decentralised optimisation methods, see the following papers and references therein.

Lower bounds for oracle and communication complexity of convex decentralised optimisation problems were obtained in \cite{scaman2019optimal}. The corresponding optimal efficiency can be achieved by the methods which are designed as follows: the requirement of consensus is represented as an affine constraint, and the variable is divided into copies among the computational nodes so that an accelerated gradient method can be applied to solve the affine-constrained convex optimisation problem by updating each copy separately. Based on the selected accelerated algorithm, different approaches have been proposed. In a deterministic case, \cite{uribe2020dual} (based on \cite{nesterov2013introductory}) and \cite{xu2020accelerated} (based on \cite{chambolle2016ergodic}), and in a stochastic case, \cite{gorbunov2019optimal} (based on \cite{gasnikov2018universal}) and \cite{lan2020communication} (based on \cite{chambolle2016ergodic}), with their corresponding large deviations probability estimates. 

An alternative stochastic algorithm which is based on \cite{devolder2011stochastic} is considered in \cite{krawtschenko2020distributed}. Its convergence analysis is quite similar to that of methods based on \cite{gasnikov2018universal}. Although its iteration takes more arithmetic operations, it is advantageous when the batch size and quantization parameters vary. It is known that variable batch size allows one to reduce sufficient iteration numbers down to the optimal one independent of stochasticity. In our analysis, we extend this to the number of samples in quantization.

\textit{Quantization.} An operation $Q_\xi$, $\xi \in \Xi$, applied to a vector $x \in S \subseteq \mathbb{R}^n$ for making the result $Q_\xi(x) \in \mathbb{R}^n$ encodable by fewer number of bits than the original vector, while introducing the smaller inexactness, is called quantization (or sparsification or dithering, which words distinguish the differences which are not important for us). This operation is helpful for lightening the communication load while solving distributed optimisation problems. If a gradient algorithm is used, the gradient is quantized before being shared with the neighbor nodes. We limit ourselves to methods of unbiased quantization, that is, those satisfying $\mathbb{E}_\xi [Q_\xi(x)] = x$, $\forall x \in S$. Inexactness introduced by a quantization is characterized by the second moment $\omega > 0$, that is $\mathbb{E}_\xi [\|Q_\xi(x) - x\|^2] \leq \omega \|x\|^2$, $\forall x \in S$.

\begin{table}[H]
    \centering
    \begin{tabular}{|l|l|c|c|}
    \hline
    Quantization & Reference & Second moment $\omega$ & Communicated bits\\
    \hline
        top $M$ & \cite{alistarh2018convergence} & $1 - M/n$ & $M (|float| + \log n)$\\
        random $M$ & \cite{li2020acceleration} & $n/M - 1$ & $M (|float| + \log n)$ \\
        $(p, s)$-dithering & \cite{horvoth2022natural} & $\alpha + \sqrt[p]{n} / 2^s$ & $|float| + n \log s$\\
        nearly optimal & \cite{albasyoni2020optimal} & $\alpha$ & $|float| + M \log n$\\\hline
        PPS & \multirow{2}{*}{This Work} & $\frac{2 (1 - \frac{1}{n})}{\mathrm{e} M}\sup_{x\in S} \frac{\|x\|_1^2}{\|x\|^2}$ & \multirow{2}{*}{$|float| + M \log n$} \\
        PPS (when $B = 1$) &  & $\frac{2 (1 - \frac{1}{n})}{\mathrm{e} M}$ & \\\hline
    \end{tabular}
    \caption{Properties of quantization operators. $n$ is dimensionality of the problem, $B$ is the maximum 1-norm of the gradient, $M = O(n)$ is number of samples. $\alpha \ll 1$ denotes the bounded second moment of the compression operator itself. Constant factors and terms in the second moment and communicated bits columns are omitted.}
    \label{tab:quant}
\end{table}
Many quantization methods are parameterised with a number of samples $M$, an increase of which reduces $\omega$ but leads to an increase in the number of bits required to encode its result. Table~\ref{tab:quant} summarizes the properties of widely-used quantization methods and compares them with the PPS quantization we propose, to be introduced in Section~\ref{sec:pps}. One can see that the dependence of the number of bits for PPS quantization is near-optimal. On the opposite, the second moment is generally as good as that of the random $M$ method, taking into account that $\sup_{x \in \mathbb{R}^n} \frac{\|x\|_1^2}{\|x\|^2} = n$. But when $S$ is such that $\sup_{x \in S} \|x\|_1 \leq B$, which is the case if $S = \Delta_n$, the second moment of PPS quantization with one sample becomes constant, and this disadvantage in comparison with other methods becomes insignificant.



\textbf{Limitations.} In this paper, we consider only the case when both domain and range spaces of an optimised function are Euclidean with the standard norms for the sake of simplicity and because of the practical significance of this case. In the optimisation method, we assume the prox-function is Euclidean because it is a preferable choice in our approach to the Wasserstein barycentre problem. Nevertheless, most of the results in this paper can be generalised to arbitrary norms and prox-function, as in sources on which we base our statements.  

\textbf{Notations.}
\begin{enumerate*}
    \item Bold symbols (like $\boldsymbol{\xi}$) denote vectors or finite sequences of values when there is a need to manipulate with their individual components ($i$th component of $\boldsymbol{\xi}$ is denoted by $\xi_i$), or sets of vector or operators acting on these sets (in Section~\ref{sec:distr}), therefore $\boldsymbol{x}^i$ means $i$th vector in the set. $\|x\|$ denotes Euclidean norm of vector $x$, $\|x\|_1$ denotes its Manhattan norm, $\|A\|_2$ denotes spectral norm of matrix $A$. All vectors are columns, that is $\mathbb{R}^n = \mathbb{R}^{n \times 1}$.
    \item Related to randomness: $\Xi$ can be treated as sample space as well as the range of every random variable. $P\left(A \gtreqless B\right)$ denotes the probability of the event that inequality between $A$, which depends on random variables, and $B$, which is fixed, holds. $\mathbb{E}_\xi[A]$ denotes expectation of $A$ which depends on $\xi$, with respect to the randomness which generates $\xi$. When there is a vector $\boldsymbol{\xi} \in \Xi^n$, components $\xi_i$ are independent. The random variable $\xi$ is drawn from categorical distribution with probabilities $P(k)$, $k \in \{1, ..., n\}$, if $\sum_{k=1}^n P(k) = 1$ and $\xi = k$ with probability $P(k)$ for every $k \in \{1, ..., n\}$. $\Delta_n = \{\boldsymbol{x}\geq 0\;\|\;\sum_{i=1}^n x_i = 1\}$ denotes a probability simplex.
    \item Related to asymptotic: $poly(f(A))$ denotes $1 + f(A) + (f(A))^2 + \dotsi + (f(A))^i$ for some fixed $i \in \mathbb{N}$. $B = O(f(A))$ means $\exists C, a \in \mathbb{R}\;\forall A \geq a\;(B \leq C \cdot f(A))$ if $A$ stands for large values like iterations number, or $\exists C, a \in \mathbb{R}\;\forall A \leq a\; (B \leq C \cdot f(A))$ otherwise, like if $A$ is an allowed probability $\delta$ of fail or target accuracy $\varepsilon$. $B = \widetilde{O}(f(A))$ means the same, but with $B \leq C \cdot f(A) \cdot O\big(poly\big(\log{\frac{1}{\varepsilon}}, \sqrt{\log \frac{1}{\delta}}\big)\big)$.
    \item Related to graphs: $G = \{1, ..., m\}$, where $m \in \mathbb{N}$, denotes the set of nodes in graph. $\mathsf{E} \subseteq G \times G$ denotes the set of edges in graph. For $i, j \in G$, $i\mathsf{E}j$ means that $(i, j) \in \mathsf{E}$.
\end{enumerate*} 

\section{PPS quantization} \label{sec:pps}
In this section, we introduce the oracle, which provides a sparse gradient approximation given the point, called the Probability-Proportional-to-Size (PPS) quantized gradient oracle. Let there be a differentiable function $f: \mathbb{R}^n \to \mathbb{R}$. In problems of our interest, some stochastic gradient oracle $g(x, \xi)$ is available, let it satisfy the following conditions $\forall x \in \mathbb{R}^n$:
\begin{enumerate}
    \item $\mathbb{E}_\xi[g(x, \xi)] = \nabla f(x)$,
    \item $\mathbb{E}_\xi[\exp(\|g(x, \xi) - \nabla f(x)\|^2 / \sigma^2)] \leq \mathrm{e}$,
    \item $\forall \xi \in \Xi, \|g(x, \xi)\|_1 \leq B$.
\end{enumerate}
We allow our methods to use mini-batching, that is, the oracle $G: (x, \boldsymbol{\xi}) \mapsto \frac{1}{r} \sum_{i=1}^r g(x, \xi_r)$ is available. Let $G(x, \boldsymbol{\xi}) = \boldsymbol{G}$, then $\boldsymbol{G} = [\boldsymbol{G}]_+ - [-\boldsymbol{G}]_+$, where $i$th component of $[\boldsymbol{v}]_+$ is defined by $\max\{0, v_i\}$. Note that to any vector $[\boldsymbol{G}]_+$ with non-negative components corresponds a categorical distribution with probabilities $P_{\boldsymbol{G}}(k) = [G_k]_+ / \|[\boldsymbol{G}]_+\|_1$, analogously the $P_{-\boldsymbol{G}}$ is defined. Let $\boldsymbol{k}, \boldsymbol{l} \in \{1, ..., n\}^M$ be samples drawn from $P_{\boldsymbol{G}}$ and $P_{-\boldsymbol{G}}$, correspondingly. PPS quantized gradient oracle is defined as follows:
\begin{equation} \label{eq:pps}
    PPS: (x, \boldsymbol{\xi}, \boldsymbol{k}, \boldsymbol{l}) \mapsto \frac{\|[\boldsymbol{G}]_+\|_1}{M} \sum_{i=1}^M \boldsymbol{e}_{k_i} - \frac{\|[-\boldsymbol{G}]_+\|_1}{M} \sum_{i=1}^M \boldsymbol{e}_{l_i},
\end{equation}
where $\boldsymbol{e}_i$ is $i$th standard basis' vector. If the reader finds it wrong to consider $\boldsymbol{k}$ and $\boldsymbol{l}$ as arguments of PPS, as they depend on $\boldsymbol{G}$, consider their components as uniformly distributed on $[0, 1]$ in arguments, so that they are transformed according to $P_{\boldsymbol{G}}$ and $P_{-\boldsymbol{G}}$ in course of PPS' evaluation. 

Note that $PPS(x, \boldsymbol{\xi}, \boldsymbol{k}, \boldsymbol{l})$ for predefined $M$ is fully determined by $2$ real numbers $\|[\boldsymbol{G}]_+\|_1$ and $\|[-\boldsymbol{G}]_+\|_1$, and $2M$ integer numbers $k_i$ and $l_i$ with values in between $1$ and $n$. Imagine a distributed optimisation setup, where one server calls another server's oracle and awaits for $PPS(x, \boldsymbol{\xi}, \boldsymbol{k}, \boldsymbol{l})$: this communication takes $2 |float| + 2 M \log_2 n$ bits overall.

\begin{lemma} \label{thm:var}
    Let $\sigma_{r,M}^2 = 50 \left(\frac{2 (1 - 1/n) B^2}{\mathrm{e} M} + \frac{\sigma^2 }{r}\right)$. Then, it holds that
    \begin{equation*}
        \forall x \in \mathbb{R}^n, \mathbb{E}_{\boldsymbol{\xi},\boldsymbol{k}, \boldsymbol{l}}[\exp(\|PPS(x, \boldsymbol{\xi}, \boldsymbol{k}, \boldsymbol{l}) - \nabla f(x)\|^2 / \sigma_{r,M}^2)] \leq \mathrm{e}.
    \end{equation*}
\end{lemma}
\begin{proof}
    By Lemma 3.1.3 \cite{dvinskikh2021decentralized}, it holds that $\mathbb{E}_{\boldsymbol{\xi}}[\exp(\|G(x, \boldsymbol{\xi}) - \nabla f(x)\|^2 / \sigma_r^2)] \leq \mathrm{e}$, where $\sigma_r^2 = 50 \sigma^2 / r$. Making use of decomposition $PPS(x, \boldsymbol{\xi}, \boldsymbol{k}, \boldsymbol{l}) - \nabla f(x) =$
    \begin{align*}
    = (\boldsymbol{G} - \nabla f(x)) &+ \|[\boldsymbol{G}]_+\|_1 \textstyle\left(\frac{1}{M} \sum_{i=1}^M \boldsymbol{e}_{k_i} - \frac{[\boldsymbol{G}]_+}{\|[\boldsymbol{G}]_+\|_1}\right)\\&- \|[-\boldsymbol{G}]_+\|_1 \left(\frac{1}{M} \sum_{i=1}^M \boldsymbol{e}_{l_i}- \frac{[-\boldsymbol{G}]_+}{\|[-\boldsymbol{G}]_+\|_1}\right),
    \end{align*}
    we see that if $\mathbb{E}_{k}[\exp(\|e_{k} - \boldsymbol{v}\|^2 / \alpha)] \leq \mathrm{e}$ holds (for any $\boldsymbol{v} \in\sigma(n)$, with $k$ drawn from $P_{\boldsymbol{v}}$) then $PPS(x, \boldsymbol{\xi}, \boldsymbol{k}, \boldsymbol{l})$ is sub-Gaussian with $\sigma_{r,M}^2 = 100 B^2 \alpha / M + \sigma_r^2$ by Proposition 2.6.1 \cite{vershynin2018high} and Lemma 3.1.3 \cite{dvinskikh2021decentralized}. 

    Let us now find $\alpha$. Firstly, $\|\boldsymbol{v} - e_k\|^2 = \sum_{i \neq k} v_i^2 + (v_k - 1)^2 = (\sum_{i=1}^n v_i^2) - v_k^2 + (v_k - 1)^2 = (\sum_{i=1}^n v_i^2) - 2v_k + 1$. Taking the expectation, $\mathbb{E}_k[\|\boldsymbol{v} - e_k\|^2] = \frac{1}{n} \sum_{k=1}^n [(\sum_{i=1}^n v_i^2) - 2v_k + 1] = (\sum_{i=1}^n v_i^2) - \frac{2}{n} \sum_{k=1}^n v_k + 1 = (\sum_{i=1}^n v_i^2) - \frac{2}{n} + 1$. By Cauchy--Schwarz, $(\sum_{i=1}^n v_i \cdot 1)^2 \leq n (\sum_{i=1}^n v_i^2)$, from where $1/n \leq (\sum_{i=1}^n v_i^2)$. Therefore, $\mathbb{E}_k [\|\boldsymbol{v} - e_k\|^2] = (\sum_{i=1}^n v_i^2) - 2/n + 1 \geq 1/n - 2/n + 1 = 1 - 1/n$. Since we strive for $\mathbb{E}_k[\|\boldsymbol{v} - e_k\|^2 / \alpha] \leq \mathrm{e}$, we have $(1 - 1/n) / \alpha \leq \mathrm{e}$, from where $\alpha = (1 - 1/n) / \mathrm{e}$, which is tight because the lower bound given by Cauchy--Schwarz becomes equality when $v_i = 1/n$ for all $i = 1, ..., n$.
\end{proof}

The next section introduces the generic smooth and convex optimisation problem of interest, for which the PPS quantization is tailored.

\section{Affine constraints} \label{sec:aff}
Below, the affine-constrained convex optimisation problem is introduced and reduced to the unconstrained statement similar to \cite{dvurechenskii2018decentralize,krawtschenko2020distributed}, addressable with the help of quantization method we had introduced in Section~\ref{sec:pps}. Let there be an optimisation problem
\begin{equation} \label{eq:prim}
    f(x_*) = \min_{\substack{x \in \mathbb{R}^n\\A x = b}} f(x),
\end{equation}
where $A \in \mathbb{R}^{m\times n}$ is positive-definite, $b \in \mathbb{R}^m$, $f: \mathbb{R}^n \to \mathbb{R}$ is convex and continuously differentiable. The following problem is dual to the former one:
\begin{equation} \label{eq:dual}
    \varphi(\lambda_*) = \min_{\lambda \in \mathbb{R}^m} \left\{\varphi(\lambda) = \lambda^\top b + f^*(-A^\top \lambda) \right\},
\end{equation}
where $\|\lambda_*\| \leq R$, and $f^*$ is a Fenchel--Legendre conjugate of $f$, defined by $f^*(\mu) = \max_{x \in \mathbb{R}^n} \{\mu^\top x - f(x)\}$. $\nabla \varphi$ is assumed to be $L$-Lipschitz continuous. Problem \eqref{eq:dual} is to be solved to obtain an approximately feasible $\varepsilon$-approximate solution $x$ to \eqref{eq:prim}, which satisfies:
\begin{equation} \label{eq:target}
    P(f(x) - f(x_*) \leq \varepsilon) > 1 - \delta,\quad P\left(\|A x - b\| \leq \frac{\varepsilon}{R}\right) > 1 - \delta
\end{equation}
for any small $\varepsilon, \delta > 0$.

We assume that only the conjugate function $f^*$ is computable and that there is a function $F^*(\lambda, \xi)$ and its gradient $\nabla F^*(\lambda, \xi)$ is available, which are zeroth- and first-order oracles for $f^*$, with the following properties:
\begin{enumerate}
    \item $\mathbb{E}_\xi [F^*(\mu, \xi)] = f^*(\xi)$,
    \item $F^*$ is smooth with respect to $\mu$,
    \item $F^*(\mu, \xi) = \max_{x \in \mathbb{R}^n} \{\mu^\top x - F(x, \xi)\}$ for some convex $F(x, \xi)$.
\end{enumerate}
The latter two provide us with means to restore the primal variable from the dual one:
\begin{equation} \label{eq:varres}
     x(\mu, \xi) = \arg \max_{x \in \mathbb{R}^n} \{\mu^\top x - F(x, \xi)\} = \nabla F^*(\mu, \xi).
\end{equation}

Note that if the latter equality is used in computations, $x(\mu, \xi)$ is unbiased. Using these assumptions, the oracle $g: (\lambda, \xi) \mapsto b - A \nabla F^*(-A^\top \lambda, \xi)$ is defined, which enjoys $\mathbb{E}_\xi [g(\lambda, \xi)] = \nabla \varphi(\lambda)$ by construction, and is required additionally to satisfy
\begin{enumerate}
    \item[4.] $\mathbb{E}_\xi[\exp(\|g(\lambda, \xi) - \nabla \varphi(\lambda)\|^2 / \sigma^2)] \leq \mathrm{e}$.
\end{enumerate}
Besides, we will need a stochastic counterpart $\Phi: (\lambda, \xi) \mapsto \lambda^\top b + F^*(-A^\top \lambda, \xi)$ of the dual function \eqref{eq:dual} and that of primal variable restoring function \eqref{eq:varres}, which is \begin{equation} \label{eq:xrestore}
    x(\mu) = \arg \max_{x \in \mathbb{R}^n} \{\mu^\top x - f(x)\} = \nabla f^*(\mu).
\end{equation}

\section{Theoretical guarantees}

\RestyleAlgo{ruled}

\begin{algorithm}[H]
\caption{Primal method}\label{alg:dev1}
\KwIn{iterations number $T$, $\boldsymbol{\alpha}$, $\boldsymbol{\beta}$, $\boldsymbol{r}$, $\boldsymbol{M}$, $PPS$ quantized gradient oracle}
\KwOut{$x_T$}
 $y_0 = 0$\;
 Sample $\boldsymbol{\xi}_0 \in \Xi^{r_t}$, $\boldsymbol{k}_0, \boldsymbol{l}_0 \in \Xi^{M_t}$\;
  $G_0 = PPS(y_0, \boldsymbol{\xi}_0, \boldsymbol{k}_0, \boldsymbol{l}_0)$\;
 $x_0 = -\frac{\alpha_0}{\beta_0} G_0$\;
 $A_0 = \alpha_0$\;
 \For{$t = 0, ..., T-1$}{
  $z_t = -\frac{1}{\beta_t} \sum_{i=0}^t \alpha_i G_i$\;
  $A_{t+1} = A_t + \alpha_{t+1}$\;
  $\tau_t = \frac{\alpha_{t+1}}{A_{t+1}}$\;
  $y_{t+1} = \tau_t z_t + (1 - \tau_t) x_t$\;
  Sample $\boldsymbol{\xi}_{t+1} \in \Xi^{r_t}$, $\boldsymbol{k}_{t+1}, \boldsymbol{l}_{t+1} \in \Xi^{M_t}$\; 
$G_{t+1} = PPS(y_{t+1}, \boldsymbol{\xi}_{t+1}, \boldsymbol{k}_{t+1}, \boldsymbol{l}_{t+1})$\;
  $\hat{y}_{t+1} = z_t -\frac{\alpha_{t+1}}{\beta_t} G_{t+1}$\;
 }
\end{algorithm}

\subsection{Primal optimisation}
This section provides the analysis of the primal algorithm for a problem without linear constraints. The assumptions on coefficients $\boldsymbol{\alpha}$ and $\boldsymbol{\beta}$ remain are described in Section~\ref{sec:pd}. When PPS quantized gradient oracle is mentioned below, it is introduced for a function $f$, as in Section~\ref{sec:pps}. Let there be an optimisation problem
\begin{equation} \label{eq:prim1}
    f(x_*) = \min_{x \in \mathbb{R}^n} f(x),
\end{equation}
where $f: \mathbb{R}^n \to \mathbb{R}$ is convex, differentiable, and $\nabla f$ is $L$-Lipschitz continuous.

\begin{lemma}[Theorem 3.1.4 \cite{dvinskikh2021decentralized}]
    If $\|y_t - x_*\| \leq R_*, \forall t = 0, ..., T$, $G_t$ is defined by Algorithm~\ref{alg:dev1}, then it holds that 
    \begin{align*}
        \min_{x \in \mathbb{R}^n} \Bigg\{\frac{\beta_T}{2} \|x\|^2 &+ \sum_{t=0}^T \alpha_t (f(y_t) + G_t^\top (x - y_t)))\Bigg\} \leq \frac{\beta_T R^2}{2} + A_T f(x_*) +\\
        &+ 2 R_* \left\|\sum_{t=0}^T \alpha_t (\nabla f(y_t) - G_t)\right\| + \sum_{t=0}^T \alpha_t (\nabla f(y_t) - G_t)^\top \mu_t.
    \end{align*}
\end{lemma}
The following theorem is similar to Theorem~\ref{thm:pdconv}, but coefficients of the terms therein are smaller than in original theorem, because there is no need to insure with high probability the transition to an average point.

\begin{theorem} \label{thm:pconv}
    For $\delta \in (0, 1)$, given the primal \eqref{eq:prim1} problem, Algorithm~\ref{alg:dev1} ensures that
    \begin{equation*}
        P(f(x_T) - f(x_*) \leq \varepsilon(T, \delta, \boldsymbol{\alpha}, \boldsymbol{\beta}, \boldsymbol{r}, \boldsymbol{M})) > 1 - \delta,
    \end{equation*}
    \begin{flalign*}
        &\text{for }\varepsilon(T, \delta, \boldsymbol{\alpha}, \boldsymbol{\beta}, \boldsymbol{r}, \boldsymbol{M}) = \frac{\beta_T R^2}{2 A_T} + \frac{C^\prime_1 R}{A_T} \sqrt{\sum_{t=0}^T \alpha_t^2 \sigma_{r_t,M_t}^2} + \frac{C^\prime_2}{A_T} \sum_{t=0}^T \frac{A_t \sigma_{r_t,M_t}^2}{\beta_t - L},&
    \end{flalign*}
    where $C^\prime_2 = 1 + \ln \frac{4}{\delta} = O(\log 1/\delta)$ and \[
    C^\prime_1 = \left(2 J(T) + \sqrt{2} - 1\right)  \left(\sqrt{2} + (\sqrt{2} + 1) \sqrt{3 \ln\frac{4}{\delta}}\right) + \sqrt{2} - 2=\] $= O(poly(\log T) \sqrt{\log 1/\delta})$.
\end{theorem}

\begin{algorithm}[H]
\caption{Primal-dual method}\label{alg:dev}
\KwIn{iterations number $T$, $\boldsymbol{\alpha}$, $\boldsymbol{\beta}$, $\boldsymbol{r}$, $\boldsymbol{M}$, $PPS$ quantized gradient oracle, restoring function $x(\lambda, \boldsymbol{\xi})$}
\KwOut{$\lambda_T$, $x_T$}
 $\mu_0 = 0$\;
 Sample $\boldsymbol{\xi}_0 \in \Xi^{r_t}$, $\boldsymbol{k}_0, \boldsymbol{l}_0 \in \Xi^{M_t}$\;
$G_0 = PPS(\mu_0, \boldsymbol{\xi}_0, \boldsymbol{k}_0, \boldsymbol{l}_0)$\;
 $\lambda_0 = -\frac{\alpha_0}{\beta_0} G_0$\;
 $A_0 = \alpha_0$\;
 \For{$t = 0, ..., T-1$}{
  $z_t = -\frac{1}{\beta_t} \sum_{i=0}^t \alpha_i G_i$\;
  $A_{t+1} = A_t + \alpha_{t+1}$\;
  $\tau_t = \frac{\alpha_{t+1}}{A_{t+1}}$\;
  $\mu_{t+1} = \tau_t z_t + (1 - \tau_t) \lambda_t$\;
  Sample $\boldsymbol{\xi}_{t+1} \in \Xi^{r_t}$, $\boldsymbol{k}_{t+1}, \boldsymbol{l}_{t+1} \in \Xi^{M_t}$\; 
  $G_{t+1} = PPS(\mu_{t+1}, \boldsymbol{\xi}_{t+1}, \boldsymbol{k}_{t+1}, \boldsymbol{l}_{t+1})$\;
  $\hat{\mu}_{t+1} = z_t -\frac{\alpha_{t+1}}{\beta_t} G_{t+1}$\;
  $\lambda_{t+1} = \tau_t \hat{\mu}_{t+1} + (1 - \tau_t) \lambda_t$\;
  $x_{t+1} = \frac{1}{A_{t+1}} \sum_{i=0}^{t+1} \alpha_i x(- A^\top \mu_i, \boldsymbol{\xi}_i)$\;
 }
\end{algorithm}

\subsection{Primal-dual optimisation} \label{sec:pd}
This section analyses the primal-dual fast gradient method \cite{devolder2011stochastic,krawtschenko2020distributed} with Euclidean prox-function $d: x \mapsto \frac{1}{2}\|x\|^2$, which is listed in Algorithm~\ref{alg:dev}. Large deviations' probabilities' estimations are provided for function value and constraints discrepancies to be minimised.

PPS quantized gradient oracle \eqref{eq:pps}, when mentioned in this section and after, approximates the gradient $\nabla \varphi$ and is constructed based on the oracle $g$ which is introduced in Section~\ref{sec:aff}. For the purpose of flexibly compensating its variance, parameters $r$ and $M$ of PPS differ from iteration to iteration of Algorithm~\ref{alg:dev}. The same notation will be used for PPS oracle's call while its parameters can be deduced from the dimensionality of its random arguments.

Hereinafter, we assume that coefficients $\boldsymbol{\alpha}$ and $\boldsymbol{\beta}$ satisfy the following conditions \cite{devolder2011stochastic}, $\forall t = 0, ..., T$:
\begin{enumerate*}
    \item $\alpha_0 \in (0, 1]$,
    \item $L < \beta_t \leq \beta_{t+1}$,
    \item $\alpha_t^2 \beta_t \leq \beta_{t-1} \sum_{i=0}^t \alpha_t$,
\end{enumerate*}
which will guide us in choosing coefficients to compensate for variance introduced by the gradient's stochasticity and quantization further.

\begin{theorem} \label{thm:pdconv}
    For $\delta \in (0, 1)$, given primal \eqref{eq:prim} and dual \eqref{eq:dual} problems, Algorithm~\ref{alg:dev} ensures that
    \begin{gather*}
        P(f(x_T) - f(x_*) \leq \varepsilon(T, \delta, \boldsymbol{\alpha}, \boldsymbol{\beta}, \boldsymbol{r}, \boldsymbol{M})) > 1 - \delta\\ P\left(\|A x_T - b\| < \frac{\varepsilon(T, \delta, \boldsymbol{\alpha}, \boldsymbol{\beta}, \boldsymbol{r}, \boldsymbol{M})}{R}\right) > 1 - \delta,
    \end{gather*}
    \begin{flalign*}
        &\text{for }\varepsilon(T, \delta, \boldsymbol{\alpha}, \boldsymbol{\beta}, \boldsymbol{r}, \boldsymbol{M}) = \frac{\beta_T R^2}{2 A_T} + \frac{C_3 R + C_4 L/\|A\|_2}{A_T} \sqrt{\sum_{t=0}^T \alpha_t^2 \sigma_{r_t,M_t}^2} + \frac{C_2}{A_T} \sum_{t=0}^T \frac{A_t \sigma_{r_t,M_t}^2}{\beta_t - L},&
    \end{flalign*}
    where $C_4 = \sqrt{2} \left(1 + \sqrt{3 \ln \frac{5}{\delta}}\right) = O(\sqrt{\log 1/\delta})$ and \[C_3 = C_1 + 2\sqrt{2} J(T) \left(1 + \sqrt{3 \ln\frac{5}{\delta}}\right)=\] $= O(poly(\log T) \sqrt{\log 1/\delta}))$.
\end{theorem}

\subsection{Choice of parameters: gradient noise is significant}
In the following, we focus on the analysis of primal-dual Algorithm~\ref{alg:dev} and associated problem statements \eqref{eq:prim}, \eqref{eq:dual}. Firstly, the parameters $\boldsymbol{\alpha}$, $\boldsymbol{\beta}$, $\boldsymbol{r}$, and $\boldsymbol{M}$ must be clarified to make an algorithm certain. There are two key principles: \begin{enumerate*}
    \item Samples numbers $\boldsymbol{r}$ and $\boldsymbol{M}$ should be chosen such that $\sigma_{r_t,M_t}$ is minimal for a given computational budget, according to Lemma \ref{thm:var}, that is $M_t = \frac{2 (1 - 1/n) B^2}{\mathrm{e} \sigma^2} r_t, \forall t = 0, ..., T$ \label{option:1b},
    \item Coefficients $\boldsymbol{\alpha}$ and $\boldsymbol{\beta}$ are to guarantee convergence rate of the order $\frac{L R^2}{T^2}$ if $\sigma_{r_t,M_t}$ is variable, or $\frac{L R^2}{T^2} + \frac{\sigma R}{\sqrt{T}}$ if $\sigma_{r_t,M_t} = \sigma_{r,M}$, according to \cite{devolder2011stochastic}, that is, $\forall t = 0, ..., T$, $\alpha_t = \frac{t + 1}{2 \sqrt{2}}$, $\beta_t = L + \frac{\sigma_{r_t,M_t} (t + 2)^{3/2}}{2^{1/4} \sqrt{3} R}$.
\end{enumerate*}

\begin{theorem} \label{th:choice}
    Given $\varepsilon > 0$, $0 < \delta < 1$, and $\boldsymbol{M}$ determined by \ref{option:1b}, Algorithm~\ref{alg:dev} can ensure \eqref{eq:target},
    \begin{enumerate}
        \item if $r_t = r \geq 1$, \newline
        in $\max\left\{2 \sqrt{3} r \sqrt{\frac{L R^2}{\varepsilon}}, 1200 \left(2^{1/4} C_2 + C_3\right)^2 \frac{\sigma^2 R^2}{\varepsilon^2}, 1200 C_4^2 \frac{\sigma^2 L^2}{\varepsilon^2 \|A\|_2^2}\right\}$, which is $ \widetilde{O}\left(\max\left\{r \sqrt{\frac{L R^2}{\varepsilon}}, \frac{\sigma^2 R^2}{\varepsilon^2}, \frac{\sigma^2 L^2}{\varepsilon^2 \|A\|_2^2}\right\}\right)$ calls of $g$ oracle, $r$ per iteration, which takes\newline
        $\widetilde{O}\left(B^2 \log n \max\left\{\frac{r}{\sigma^2} \sqrt{\frac{L R^2}{\varepsilon}}, \frac{R^2}{\varepsilon^2}, \frac{L^2}{\varepsilon^2 \|A\|_2^2}\right\}\right)$ bits,
        \item if $r_t = \max\left\{1, \frac{34 \sigma^2 \alpha_t}{\varepsilon L} \max\left\{18 C_2^2, (C_3 + C_4 L / (\|A\|_2 R))^2\right\}\right\}$,\newline
        in $\sum_{t=0}^T r_t = \max\left\{3^{1/2} \sqrt{\frac{L R^2}{\varepsilon}}, 900 C_2^2 \frac{\sigma^2 R^2}{\varepsilon^2}, 50 (C_3 + C_4 L / (\|A\|_2 R))^2 \frac{\sigma^2 R^2}{\varepsilon^2}\right\} = \widetilde{O}\left(\max\left\{\sqrt{\frac{L R^2}{\varepsilon}}, \frac{\sigma^2 R^2}{\varepsilon^2}, \frac{\sigma^2 L^2}{\varepsilon^2 \|A\|_2^2}\right\}\right)$ calls of $g$ oracle, $T = \sqrt{\frac{3 L R^2}{\varepsilon}}$ iterations, which takes\newline
        $\widetilde{O}\left(B^2 \log n \max\left\{\frac{1}{\sigma^2} \sqrt{\frac{L R^2}{\varepsilon}}, \frac{R^2}{\varepsilon^2}, \frac{L^2}{\varepsilon^2 \|A\|_2^2}\right\}\right)$ bits.
    \end{enumerate}
\end{theorem}

\begin{proof}
    Let us begin with the case $r_t = r$. Since $M_t = \frac{2 (1 - 1/n) B^2}{\mathrm{e} \sigma^2} r_t$, $\sigma^2_{r_t, M_t} = 100 \sigma^2 / r$. By Theorem~\ref{thm:pdconv},
    \begin{equation*}
        \frac{4 \sqrt{2} L R^2}{T^2} + \frac{20 C_3 \sigma R}{\sqrt{3 r T}} + \frac{20 C_4 \sigma L / \|A\|_2}{\sqrt{3 r T}} + \frac{200 C_2 \sigma^2}{3 r \sqrt{T}} \frac{2^{1/4} \sqrt{3 r} R}{10 \sigma} < \varepsilon
    \end{equation*}
    ensures \eqref{eq:target}, which implies that it is sufficient to make
    \begin{equation*}
        T > \max\left\{2 \sqrt{3} \sqrt{\frac{L R^2}{\varepsilon}}, \frac{3 \cdot 400 (C_3 + 2^{1/4} C_2)^2}{r} \frac{\sigma^2 R^2}{\varepsilon^2}, \frac{9 \cdot 400 C_4^2}{3 r} \frac{\sigma^2 L^2}{\varepsilon^2 \|A\|_2^2}\right\},
    \end{equation*}
    each one calls $g$ oracle $r$ times. To find the number of bits needed to describe that number of quantized oracle's call results, one needs to divide it by $r$ (which corresponds to the summation of results of $r$ separate calls of $g$) and multiply it by $2 M \log_2 n$, i.e., by $\frac{4 (1 - 1/n) B^2 r}{\mathrm{e} \sigma^2} \log_2 n$ (which corresponds to the quantization).

    Let us move now to the second case. Using the formula for $r_t$, in the case $r_t > 1$, the same relation between $r_t$ and $M_t$, and the fact that $\sigma_{r_t, M_t}^2 = 100 \sigma^2 / r_t$, one has $\alpha_t^2 \sigma_{r_t,M_t}^2 \leq \frac{50 \varepsilon L R^2}{17 (C_3 R + C_4 L / \|A\|_2)^2} \cdot \alpha_t \Rightarrow \frac{C_3 R + C_4 L/\|A\|_2}{A_T} \sqrt{\sum_{t=0}^T \alpha_t^2 \sigma_{r_t,M_t}^2} = O\left(\sqrt{\varepsilon} \frac{\sqrt{L R^2}}{T}\right)$, and $\frac{2^{1/4} \sqrt{3} C_2 R}{A_T} \sum_{t=0}^T \frac{A_t \sigma_{r_t, M_t}}{(t+2)^{3/2}} \leq \sqrt{\varepsilon} \cdot \frac{5 \cdot 2^{1/4} \sqrt{L R^2}}{\sqrt{51} A_T} \sum_{t=0}^T \frac{A_t}{\sqrt{\alpha_t} (t+2)^{3/2}} = O\left(\sqrt{\varepsilon} \frac{\sqrt{L R^2}}{T}\right)$. After substituting the suggested $T$, we find that \eqref{eq:target} holds. 
\end{proof}

\subsection{Choice of parameters: gradient noise is negligible}
In the case of $\sigma \ll 0$, the upper bounds on the complexity of Algorithm~\ref{alg:dev} in the previous section result in rough estimates of the order $1/\sigma^2$, which is due to the choice of $M_t = \frac{2 (1 - 1/n) B^2}{\mathrm{e} \sigma^2} r_t, \forall t = 0, ..., T$. Instead, we make $M_t$ variable in this section and bind $r_t$ to it using the relation $r_t = \frac{\mathrm{e} \sigma^2}{2 (1 - 1/n) B^2} M_t$ \label{option:3a}. 

\begin{theorem} \label{th:choice2}
    Given $\varepsilon > 0$, $0 < \delta < 1$, and $\boldsymbol{r}$ determined by \ref{option:3a}, Algorithm~\ref{alg:dev} can ensure \eqref{eq:target},
    \begin{enumerate}
        \item if $M_t = M \geq \frac{2 (1 - 1/n) B^2}{\mathrm{e} \sigma^2}$, \newline
        in $\sigma^2 \max\left\{\frac{\mathrm{e} M \sqrt{3}}{(1 - 1/n) B^2} \sqrt{\frac{L R^2}{\varepsilon}}, 1200 \left(2^{1/4} C_2 + C_3\right)^2 \frac{R^2}{\varepsilon^2}, 1200 C_4^2 \frac{L^2}{\varepsilon^2 \|A\|_2^2}\right\}$, which is $ \widetilde{O}\left(\sigma^2 \max\left\{\frac{M}{B^2} \sqrt{\frac{L R^2}{\varepsilon}}, \frac{R^2}{\varepsilon^2}, \frac{L^2}{\varepsilon^2 \|A\|_2^2}\right\}\right)$ calls of $g$ oracle, which takes\newline
        $\widetilde{O}\left(B^2 \log n \max\left\{ \frac{M}{B^2} \sqrt{\frac{L R^2}{\varepsilon}}, \frac{R^2}{\varepsilon^2}, \frac{L^2}{\varepsilon^2 \|A\|_2^2}\right\}\right)$ bits,
        \item if $M_t = M < \frac{2 (1 - 1/n) B^2}{\mathbb{e} \sigma^2}, M \geq 1$, \newline
        in $\frac{(1 - 1/n) B^2}{2 \mathrm{e} M} \max\left\{\frac{4 \mathrm{e} M \sqrt{3}}{(1 - 1/n) B^2} \sqrt{\frac{L R^2}{\varepsilon}}, 1200 \left(2^{1/4} C_2 + C_3\right)^2 \frac{R^2}{\varepsilon^2}, 4800 C_4^2 \frac{L^2}{\varepsilon^2 \|A\|_2^2}\right\} = \widetilde{O}\left(\frac{B^2}{M} \max\left\{\frac{M}{B^2} \sqrt{\frac{L R^2}{\varepsilon}}, \frac{R^2}{\varepsilon^2}, \frac{L^2}{\varepsilon^2 \|A\|_2^2}\right\}\right)$ calls of $g$ oracle, which takes\newline
        $\widetilde{O}\left(B^2 \log n \max\left\{ \frac{M}{B^2} \sqrt{\frac{L R^2}{\varepsilon}}, \frac{R^2}{\varepsilon^2}, \frac{L^2}{\varepsilon^2 \|A\|_2^2}\right\}\right)$ bits,
       \item if $M_t = \max\left\{1, \frac{68 (1 - 1/n) B^2 \alpha_t}{\varepsilon \mathrm{e} L} \max\left\{18 C_2^2, (C_3 + C_4 L / (\|A\|_2 R))^2\right\}\right\}$,
       \newline
        in $\sum_{t=0}^T r_t = \max\left\{3^{1/2} \sqrt{\frac{L R^2}{\varepsilon}}, 900 C_2^2 \frac{\sigma^2 R^2}{\varepsilon^2}, 50 (C_3 + C_4 L / (\|A\|_2 R))^2 \frac{\sigma^2 R^2}{\varepsilon^2}\right\} = \widetilde{O}\left(\max\left\{\sqrt{\frac{L R^2}{\varepsilon}}, \frac{\sigma^2 R^2}{\varepsilon^2}, \frac{\sigma^2 L^2}{\varepsilon^2 \|A\|_2^2}\right\}\right)$ calls of $g$ oracle, $T = \sqrt{\frac{3 L R^2}{\varepsilon}}$ iterations, which takes\newline
        $\sum_{t=0}^T M_t = \widetilde{O}\left(\log n \max\left\{\sqrt{\frac{L R^2}{\varepsilon}}, \frac{B^2 R^2}{\varepsilon^2}, \frac{ B^2 L^2}{\varepsilon^2 \|A\|_2^2}\right\}\right)$ bits.
    \end{enumerate}
\end{theorem}

\begin{proof}
    Let us begin with the case $M_t = M \geq \frac{2 (1 - 1/n) B^2}{\mathrm{e} \sigma^2}$. Since $r_t = \frac{\mathrm{e} \sigma^2}{2 (1 - 1/n) B^2} M_t$, $\sigma^2_{r_t, M_t} = 200 (1 - 1/n) \frac{B^2}{\mathrm{e} M}$. By Theorem~\ref{thm:pdconv},
    \begin{gather*}
        \frac{4 \sqrt{2} L R^2}{T^2} + \frac{20 \sqrt{2} \sqrt{1 - 1/n} C_3 B R}{\sqrt{3 \mathrm{e} M T}} + \frac{20 \sqrt{2} \sqrt{1 - 1/n} C_4 B L / \|A\|_2}{\sqrt{3 \mathrm{e} M T}} +\\
        + \frac{400 (1 - 1/n) C_2 B^2}{3 \mathrm{e} M \sqrt{T}} \frac{2^{1/4} \sqrt{3 \mathrm{e} M} R}{10 \sqrt{2} \sqrt{1 - 1/n} B} < \varepsilon
    \end{gather*}
    ensures \eqref{eq:target}, which implies that it is sufficient to make
    \begin{equation*}
        T > \max\left\{2 \sqrt{3} \sqrt{\frac{L R^2}{\varepsilon}}, \frac{3 \cdot 800 (1 - 1/n) (C_3 + 2^{1/4} C_2)^2}{\mathrm{e}M} \frac{B^2 R^2}{ \varepsilon^2}, \frac{9 \cdot 800 (1 - 1/n) C_4^2}{3 \mathrm{e} M} \frac{B^2 L^2}{\varepsilon^2 \|A\|_2^2}\right\},
    \end{equation*}
    each one calls $g$ oracle $r = \frac{\mathrm{e} \sigma^2}{2 (1 - 1/n) B^2} M$ times. To find the number of bits needed to describe that number of quantized oracle call results, one needs to divide it by $r$ (which corresponds to the sum of the results of the separate calls of $r$ of $g$) and multiply it by $2 M \log_2 n$, that is, by $\frac{4 (1 - 1/n) B^2 r}{\mathrm{e} \sigma^2} \log_2 n$ (which corresponds to quantization), taking into account $M \sigma^2 \geq 2 (1 - 1/n) B^2 / \mathrm{e}$. In the case of $M < \frac{2 (1 - 1/n) B^2}{\mathrm{e} \sigma^2}$ all the terms of the final estimates remain the same, although the value of $\sigma_{r_t, M_t}$ depends on both $\sigma$ and $M$ in that case: one can see that the only term that changes due to that reason is one proportional to $R^2$, for which the sufficient condition to ensure \eqref{eq:target} looks like $\frac{20 \sqrt{2} \sqrt{1 - 1/n} C_4 B R}{\sqrt{3 \mathrm{e} M T}} + \frac{80 \cdot 2^{1/4} C_2 B R}{\sqrt{3 \mathrm{e} M T}} \frac{1}{ \sqrt{2 \mathrm{e} \sigma^2 M / (B^2 (1 - 1/n)) + 4}} < \varepsilon$ now,
    which in the worst case is similar to the old one, if we use $\sqrt{a + b} \leq \sqrt{a} + \sqrt{b}$ and that $2 \mathrm{e} \sigma^2 M / (B^2 (1 - 1/n)) < 4$.

    Let us move now to the third case. Using the formula for $M_t$, the same relation between $M_t$ and $r_t$, and the fact that $\sigma_{r_t, M_t}^2 = 200 (1 - 1/n) B^2 / (\mathrm{e} M_t)$, one has $\alpha_t^2 \sigma_{r_t,M_t}^2 \leq \frac{50 \varepsilon L R^2}{17 (C_3 R + C_4 L / \|A\|_2)^2} \cdot \alpha_t \Rightarrow \frac{C_3 R + C_4 L/\|A\|_2}{A_T} \sqrt{\sum_{t=0}^T \alpha_t^2 \sigma_{r_t,M_t}^2} = O\left(\sqrt{\varepsilon} \frac{\sqrt{L R^2}}{T}\right)$, and $\frac{2^{1/4} \sqrt{3} C_2 R}{A_T} \sum_{t=0}^T \frac{A_t \sigma_{r_t, M_t}}{(t+2)^{3/2}} \leq \sqrt{\varepsilon} \cdot \frac{5 \cdot 2^{1/4} \sqrt{L R^2}}{\sqrt{51} A_T} \sum_{t=0}^T \frac{A_t}{\sqrt{\alpha_t} (t+2)^{3/2}} = O\left(\sqrt{\varepsilon} \frac{\sqrt{L R^2}}{T}\right)$. After substituting the suggested $T$, we find that \eqref{eq:target} is valid. 
\end{proof}

Let us focus on the first case described in Theorem~\ref{th:choice2}, because it covers more values of $M$, which makes it useful in practice, because increasing $M$ is computationally free, while it makes the algorithm more stable. One can notice that the estimate of the number of oracle calls in that case is proportional to $\sigma^2$, which can be very small and comparable to $\varepsilon$, so one may wonder if this leads to improvement of the asymptotic estimate of the complexity of the problem in terms of its dependence on $\varepsilon$. Let us examine the following cases:
\begin{enumerate}
    \item $\sigma = \Theta(\varepsilon^{1/2})$. Then $M$ must be $\Omega(1 / \varepsilon)$. The estimate of oracle complexity becomes $\widetilde{O}\left( \max\left\{M \sqrt{\varepsilon} \cdot  \frac{\sqrt{L R^2}}{B^2}, \frac{1}{\varepsilon} \cdot \max\{R, L/\|A\|_2\}\right\}^2\right)$. If $M = \Theta(1 / \varepsilon)$, the iteration complexity term, proportional to $\sqrt{L R^2}$, takes its standard form, while the terms introduced by the noise in gradient are now $O(1 / \varepsilon)$ instead of $O(1 / \varepsilon^2)$. One should notice that the estimate with the improved order of $\varepsilon$ (which can even contradict the lower bound on the complexity of the problem) holds only for $\varepsilon = \Theta(1/n)$. Therefore, its better to specialize in by saying that accuracy $1/n$ can be achieved in $\widetilde{\Theta}\left( \max\left\{\frac{\sqrt{n L R^2}}{B^2}, n \max\{R, L /\|A\|_2\}^2\right\}\right)$ oracle calls,~--- let us call it ``specific complexity'' of the problem. One can further specialize this estimate depending on properties of the oracle; for example, if the gradient lies in the $\ell_2$-ball, $B$ is proportional to $\sqrt{n}$, the iteration complexity term ceases to depend on $n$, which slightly improves the estimate. 
    \item $\sigma = \Theta(\varepsilon)$. Therefore, $M = \Omega(1 / \varepsilon^2)$, and the estimate becomes $\widetilde{O}\left( \max\left\{M \varepsilon^{3/2} \frac{\sqrt{L R^2}}{B^2}, \max\{R^2, \frac{L^2}{\|A\|_2^2}\}\right\}\right)$. If $M = \Theta(1/\varepsilon^2)$ improves the complexity of the problem only for $\varepsilon = \Omega(1/\sqrt{n})$, which might be useful in some applications but does not allow us to assess the algorithm using the notion of specific complexity.
\end{enumerate}

\RestyleAlgo{ruled}
\begin{algorithm}[H]
\caption{Decentralised primal-dual fast gradient method}\label{alg:devdec}
\KwIn{iterations number $T$, $\boldsymbol{\alpha}$, $\boldsymbol{\beta}$, $\boldsymbol{r}$, $\boldsymbol{M}$, $(PPS^i)_{i \in G}$ quantized gradient oracle, primal variable restoring function $x^i(\lambda, \boldsymbol{\xi})$ for every $i \in G$, Laplacian matrix $W$}
\KwOut{$\lambda_T$, $x_T$}
\begin{minipage}{0.44\textwidth}
 \For{{\normalfont \textbf{each $i$th node}}}{
     $\mu_0^i = 0$\;
     Sample $\boldsymbol{\xi}_0^i \in \Xi^{r_t}$, $\boldsymbol{k}_0^i, \boldsymbol{l}_0^i \in \Xi^{M_t}$\;
    $G_0^i = PPS^i(\mu_0^i, \boldsymbol{\xi}_0^i, \boldsymbol{k}_0^i, \boldsymbol{l}_0^i)$\;
     Send $G_0^i$ to each $j \mathsf{E} i$\;
 }
 \For{{\normalfont \textbf{each $i$th node}}}{
     $\lambda_0^i = -\frac{\alpha_0}{\beta_0} \sum_{j\mathsf{E} i} W_{ij} G_0^j$\;
 }
 $A_0 = \alpha_0$\;
 See next column\;
 \vspace{0.4cm}
\end{minipage}
\begin{minipage}{0.44\textwidth}
 \For{$t = 0, ..., T-1$}{
  $A_{t+1} = A_t + \alpha_{t+1}$\;
  $\tau_t = \frac{\alpha_{t+1}}{A_{t+1}}$\;
 \For{{\normalfont \textbf{each $i$th node}}}{
  $z_t^i = -\frac{1}{\beta_t} \sum_{k=0}^t \alpha_k \sum_{j\mathsf{E} i} W_{ij} G_k^j$\;
  $\mu_{t+1}^i = \tau_t z_t^i + (1 - \tau_t) \lambda_t^i$\;
  Sample $\boldsymbol{\xi}_{t+1}^i \in \Xi^{r_t}$, $\boldsymbol{k}_{t+1}^i, \boldsymbol{l}_{t+1}^i \in \Xi^{M_t}$\; 
  $G_{t+1}^i = PPS^i(\mu_{t+1}^i, \boldsymbol{\xi}_{t+1}^i, \boldsymbol{k}_{t+1}^i, \boldsymbol{l}_{t+1}^i)$\;
  Send $G_{t+1}^i$ to each $j \mathsf{E} i$\;
  }
 \For{{\normalfont \textbf{each $i$th node}}}{
  $\hat{\mu}_{t+1}^i = z_t^i -\frac{\alpha_{t+1}}{\beta_t} \sum_{j\mathsf{E} i} W_{ij} G_{t+1}^j$\;
  $\lambda_{t+1}^i = \tau_t \hat{\mu}_{t+1}^i + (1 - \tau_t) \lambda_t^i$\;
  $x_{t+1}^i = \frac{\alpha_{t+1}}{A_{t+1}} x^i(\mu_i, \boldsymbol{\xi}_i) + \frac{A_t}{A_{t+1}} x_t^i$\;
  }
 }
\end{minipage}
\end{algorithm}

\section{Distributed setup} \label{sec:distr}
The decentralised optimisation task can be considered as an affine-constrained optimisation problem, described in Section~\ref{sec:aff}, in a way that we study further. Network of a distributed computing system is presented by a connected graph $(G, \mathsf{E})$, where the number of nodes is $|G| = m$ and $\mathsf{E} = \{(i, j)\;|\;i\text{ and }j\text{ are neighbours}\}$, $\mathsf{E}$ is symmetric (graph is undirected). Each $i$th node has its own function $f_i$, variable $x^i$, they are tabulated in $\boldsymbol{x} = (x^1\;\dots\;x^m)$. The Laplacian matrix $W \in \mathbb{R}^{m \times m}$ is associated with the graph, we introduce a corresponding averaging operator $\boldsymbol{W} = W \otimes I_n$, where $I_n \in \mathbb{R}^{n\times n}$ is identity and $\otimes$ is a Kronecker product. The smallest eigenvalue of $W$ is zero, the second smallest is $\lambda_2(W) > 0$  \cite{mohar1991eigenvalues}, the largest is $\|W\|_2$. The desired solution for a decentralised optimisation problem assumes that it is approximately consensus, that is $x^1 = \dots = x^m$, which is equivalent to $\sqrt{\boldsymbol{W}} \boldsymbol{x} = 0$ \cite{scaman2017optimal}. Finally, let there be a decentralised optimisation problem
\begin{equation*}
    f(\boldsymbol{x}_*) = \min_{\stackrel{x^1, ..., x^m \in \mathbb{R}^n}{\sqrt{\boldsymbol{W}}\boldsymbol{x} = 0}} \left\{ f(\boldsymbol{x}) = \frac{1}{m} \sum_{i=0}^m f_i(x^i) \right\}
\end{equation*}
which is of the form \eqref{eq:prim}. Additionally, $f$ is assumed to be $\frac{\gamma}{m}$-strongly convex. The corresponding dual problem is 
\begin{equation} \label{eq:dual_distr}
    \varphi(\sqrt{\boldsymbol{W}}\boldsymbol{\lambda}_*) = \min_{\lambda^1, ..., \lambda^m \in \mathbb{R}^n} \left\{ \varphi(\sqrt{\boldsymbol{W}}\boldsymbol{\lambda}) = \frac{1}{m} \sum_{i=0}^m \varphi_i(m (\sqrt{\boldsymbol{W}} \boldsymbol{\lambda})^i) \right\},
\end{equation}
where $\varphi_i = f_i^*$. By Lemma 2, \cite{dvurechenskii2018decentralize}, \eqref{eq:dual_distr} has $L$-Lipschitz continuous gradient, $L = \frac{m \|W\|_2}{\gamma}$, which is $\nabla \varphi(\sqrt{\boldsymbol{W}} \boldsymbol{\lambda}) = \sqrt{\boldsymbol{W}} \boldsymbol{x}(\sqrt{\boldsymbol{W}} \boldsymbol{\lambda})$, where $\boldsymbol{x}(\cdot)$ is \eqref{eq:xrestore} but generalised to tabulated $\boldsymbol{x}$. Hereinafter, we denote $\lambda = \sqrt{\boldsymbol{W}} \boldsymbol{\lambda}$, and similarly for all plain counterparts of bold variables: this replacement is convenient for distributed implementation because
\begin{equation*}
    \sqrt{\boldsymbol{W}} (\boldsymbol{\lambda} - \nabla \varphi(\sqrt{\boldsymbol{W}} \boldsymbol{\lambda})) = \lambda - \boldsymbol{W} \boldsymbol{x}(\sqrt{\boldsymbol{W}} \boldsymbol{\lambda}) = \lambda - \boldsymbol{W} \boldsymbol{x}(\lambda),
\end{equation*}
which is used in Algorithm~\ref{alg:devdec}. Given the $\sigma_i$-sub-Gaussian stochastic oracle $\nabla F_i^*(\lambda^i, \xi)$ for $\nabla \varphi_i(\lambda^i)$, it holds that $\nabla \varphi(\lambda, \xi)$ is sub-Gaussian with $\sigma^2 = O(\|W\|_2 m \sigma_i^2)$, by Lemma 3.2.1, \cite{dvinskikh2021decentralized}. The PPS quantized oracle $PPS^i$ is introduced for each $i$th node based on $\nabla F_i^*(\lambda^i, \xi)$. By Theorem 3.2.3, \cite{dvinskikh2021decentralized}, $\|\lambda_*\|^2 \leq R^2 = \frac{B_*^2}{m \lambda_2(W)}$, where $B_*$ is such that $\|\nabla f_i(x^*)\| \leq B_*$. For brevity, we denote $\chi(W) = \frac{\|W\|_2}{\lambda_2(W)}$. $d_j$ denotes the degree of $j$th node, $d = \max_{j\in G} d_j$. $D$ denotes the diameter of the graph.\newline
\begin{cor} \label{cor:main}
    Given $\varepsilon > 0$, $0 < \delta < 1$, $\boldsymbol{r}$ determined by \eqref{eq:r}, and $\boldsymbol{M}$ determined by \ref{option:1b}, Algorithm~\ref{alg:devdec} can ensure $P(f(\boldsymbol{x}) - f(\boldsymbol{x}_*) \leq \varepsilon) > 1 - \delta$ and $P\left(\|\sqrt{\boldsymbol{W}} \boldsymbol{x}\| \leq \frac{\varepsilon \sqrt{2 m}}{B_*}\right) > 1 - \delta$, if any node sends 
    \begin{flalign*}
        &\widetilde{O}\left(B^2 d \log n \max\left\{\frac{1}{\sigma_i^2} \sqrt{\frac{D B_*^2}{\gamma \varepsilon m d}},\frac{D B_*^2}{\varepsilon^2}, \frac{m^2}{\gamma^2 \varepsilon^2}\right\}\right)\text{ bits.}&
    \end{flalign*}
\end{cor}
\begin{proof}
    It follows from upper and lower bounds on $\|W\|_2$ and $\lambda_2$ from \cite{fiedler1973algebraic,anderson1985eigenvalues,mohar1991eigenvalues}.
\end{proof}

\section{Application example}

\subsection{WB problem}
Let there be a network where $i$th node can draw samples from the distribution $\mu^i$, which belongs to the set of continuous probability distributions $\mathcal{P}(X)$ over metric space $X$ and has density function $q^i(x)$. The task is to collaboratively compute the entropy-regularized semi-discrete Wasserstein barycentre, which is discrete probability distribution $\nu$, that is $\nu = \sum_{i=1}^n \nu_i \delta(y_i)$ for some $y_1, ..., y_n \in Y$, and $\nu \in \sigma(n)$ minimises $\frac{1}{m} \sum_{i=1}^{m} \mathcal{W}_{\gamma,\mu^i}(\nu)$,
where $\mathcal{W}_{\gamma,\mu}(\nu)$ is the entropy-regularized semi-discrete Wasserstein distance
\begin{align}\label{WassDis}
\mathcal{W}_{\gamma,\mu}(\nu) = \min_{\pi \in \Pi(\mu,\nu)}\left\{\sum_{i=1}^n \int_X c(y_i, x) \pi_i(x) \diff x +  \gamma \sum_{i=1}^n \int_X
\pi_i(x) \log \frac{\pi_i(x)}{\xi} \diff x\right\},
\end{align}
where $u$ is the density of uniform distribution on $X \times Y$, $c(y, x)$ denotes the (continuous) cost of transporting a unit of mass from point $y \in Y$ to point $x \in X$, $\Pi(\mu,\nu)$ is the set of admissible coupling measures from $\mathcal{P}(X) \times \sigma(n)$ which satisfy $\sum_{i=1}^n \pi_i(x) = q(x), \forall x \in X$ and $\int_{X} \pi_i(x)\diff x = \nu_i, \forall i \in \{1, ..., n\}$,
and $\gamma \geq 0$ is the regularization parameter. Note that, unlike~\cite{genevay2016stochastic}, we regularize the problem by the Kullback--Leibler divergence between $\nu$ and uniform distribution $\xi$, which gives us an explicit formula for the conjugate $\mathcal{W}^*_{\gamma,\mu}$. According to Section~\ref{sec:distr}, original problem is equivalent to $\max_{\nu^1,\dots, \nu^m \in \sigma(n), \sqrt{\boldsymbol{W}} \boldsymbol{\nu}=0} -\frac{1}{m}\sum_{i=1}^{m} \mathcal{W}_{\gamma}(\mu^i, p^i)$,
and its dual is the minimisation of $\mathcal{W}_{\gamma,\boldsymbol{\mu}}^*(\boldsymbol{\lambda}) = \frac{1}{m}\sum_{i=1}^{m} \mathcal{W}^*_{\gamma,\mu^i}(m[\sqrt{\boldsymbol{W}}\boldsymbol{\lambda}]_i)$. 
Similarly, we denote $\lambda = m[\sqrt{\boldsymbol{W}}\boldsymbol{\lambda}]$.
\begin{lemma}[Lemma 1, \cite{dvurechenskii2018decentralize}] \label{lm:Lipschitz_gradient} 
For any positive Radon probability measure $\mu \in \mathcal{P}(X)$ with density $q(x)$ and random variable $x \sim \mu$, it holds that
\begin{gather*}
\mathcal{W}_{\gamma,\mu}^*(\lambda) = \mathbb{E}_{x}\left[ \gamma
\log \frac{\sum_{i=1}^n\exp \frac{\lambda_i - c(z_i, x)}{\gamma}}{q(x)}\right]\\ [\nabla \mathcal{W}_{\gamma,\mu}^*(\lambda)]_i = \mathbb{E}_{x}
 \left[\frac{\exp \frac{\lambda_i-c(z_i, x)}{\gamma}}{\sum_{j=1}^n\exp\frac{\lambda_j-c(z_j, Y)}{\gamma}}\right],
\end{gather*}
and $\nabla \mathcal{W}_{\gamma,\mu}^*(\lambda)$ is Lipschitz continuous with 
$L = \frac{m}{\gamma}$.
\end{lemma}

\subsection{Numerical experiment}
All the results in this section were obtained by a simulation of distributed execution performed on a single personal computer with the Apple M2 chip with an 8-core CPU and 8GB of unified memory. Source code is available at \url{https://github.com/PPS-quantization/Wasserstein-barycentre}.

In the first experiment, distributions $\mu^i$ are chosen to be Gaussian with randomly generated mean and variance parameters. We compare our approach with the non-quantized accelerated gradient method (AGM) \cite{dvurechenskii2018decentralize}. Figure~\ref{fig:gaussians} shows curves of the consensus gap and the value of the dual function, associated with the WB problem, for approximate solutions given by Algorithm~\ref{alg:devdec}, with different sampling schemes parameterised by sequences $M_t$ and $r_t$. Figures~\ref{fig:gaussians2},\ref{fig:gaussians3} consider networks with different topology. Figure~\ref{fig:gaussians2vis} shows the visualised approximate barycentric Gaussians found by Algorithm~\ref{alg:devdec} after $k=\{10,100,200,500\}$ iterations. Predictably, the algorithm from \cite{dvurechenskii2018decentralize} demonstrates the best performance because of non-quantized communication between nodes. At the same time, one can see that even with a small number of samples, that is $M_t=1$ and $r_t=10$, Algorithm~\ref{alg:devdec} can rapidly converge to the barycentre with significantly less communication overhead. 

Secondly, we carried out an experiment with the MNIST dataset \cite{deng2012mnist}. $40$ images of the digit ``2'' with dimensions $100\times100$ (vectorised so that $n = 10^4$) define $\mu^i$. The network is defined by a random Erd\H{o}s--R\'enyi graph. Figure~\ref{fig:bad} shows the visualised approximate barycentre found by AGM \cite{dvurechenskii2018decentralize}, and by  Algorithm~\ref{alg:devdec} after 5000 iterations for $M_t = r_t =100$ (all nodes are shown) and other values of $M_t$ and $r_t$. One can see that Algorithm~\ref{alg:devdec} achieves similar quality of the approximate barycentre, when using constant batch size $r_t$ and much smaller number $M_t \ll n$ of bits sent in each communication round than AGM. 

\begin{figure}[H]
    \centering
    \subfigure[Consensus gap and value of the dual function. Erd\H{o}s--R\'enyi random graph. Constant sampling scheme]{%
        \includegraphics[trim=9 7 12 5,clip,width=0.32\linewidth]{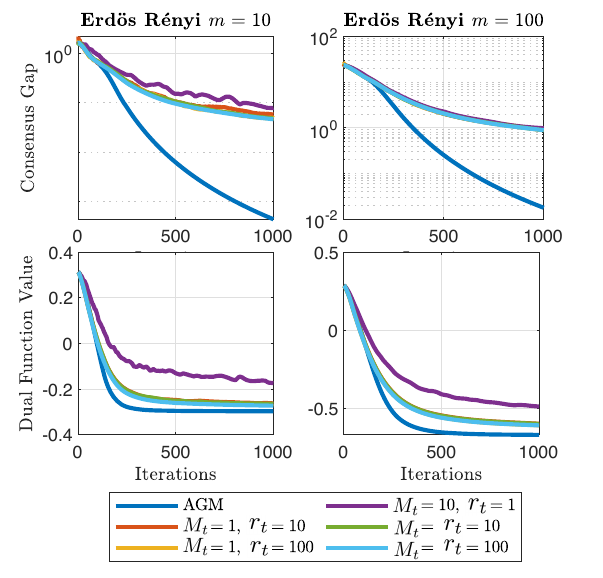}
        \label{fig:gaussians}
    }
    \subfigure[Constraints discrepancy and value of the dual function. Different graph topology with $m=30$. Constant sampling scheme]{%
        \includegraphics[trim=35 20 43 20,clip,width=0.32\linewidth]{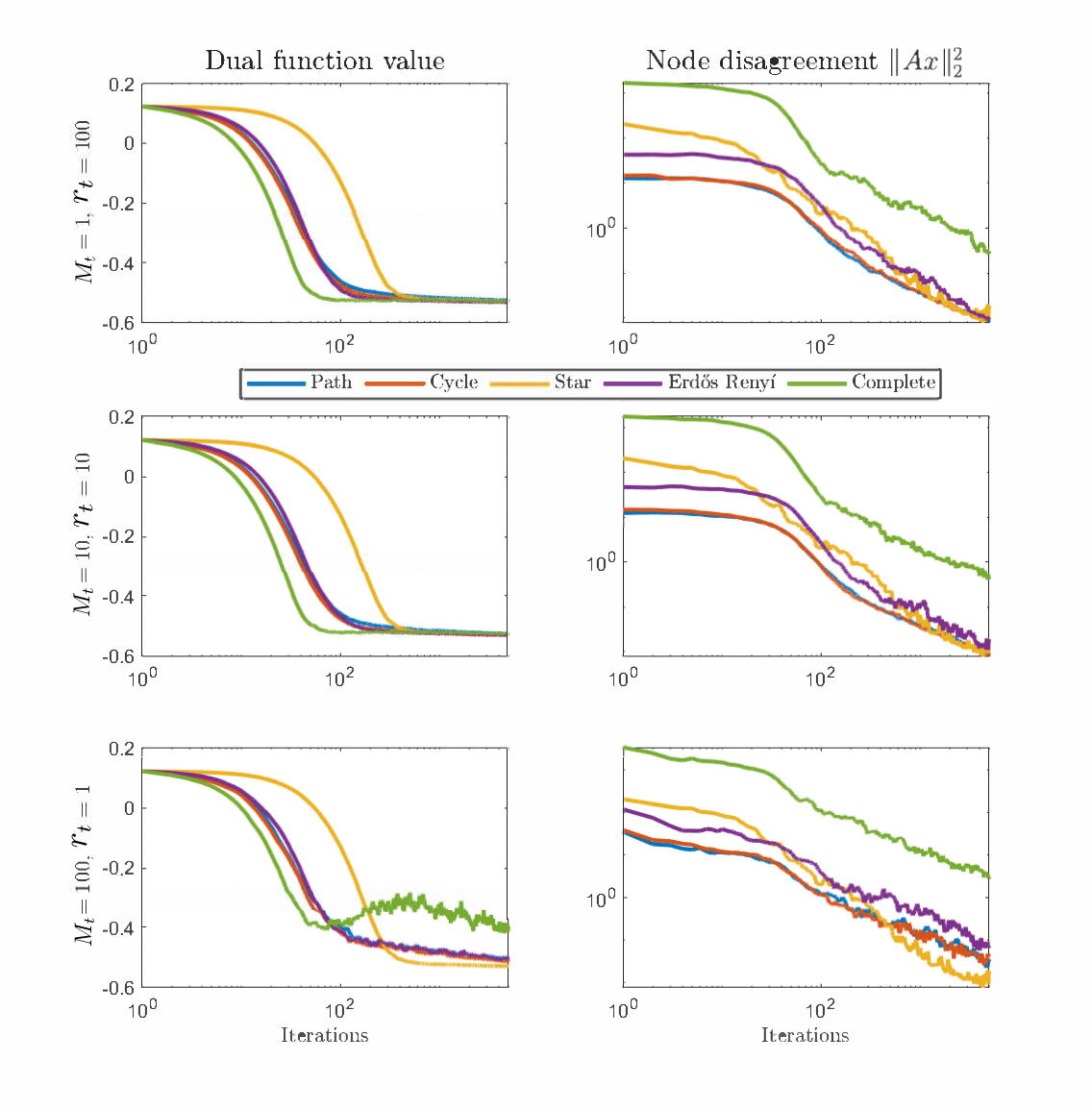}
        \label{fig:gaussians2}
    }
    \subfigure[Consensus gap and value of the dual function. Different graph topology. Time-dependent sampling $M_t = r_t = O(t)$]{%
       \includegraphics[width=0.29\linewidth]{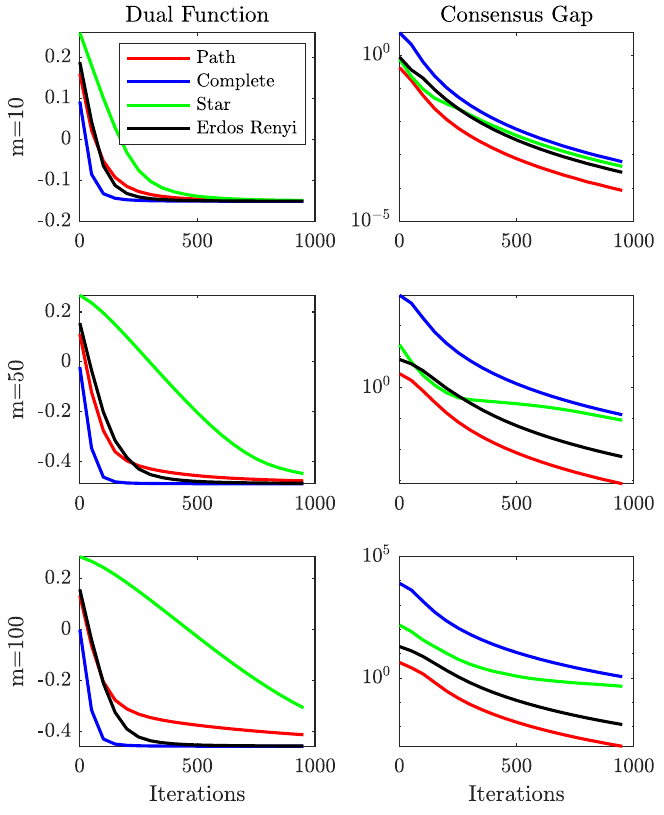}
        \label{fig:gaussians3}
    }

    \caption{Convergence curves of Algorithm~\ref{alg:devdec} and AGM for different network topology, number of nodes, and sampling schemes}
    \label{fig:combined}
\end{figure}
\begin{figure}[H]
    \centering
    \subfigure[MNIST $100\times100$ images. 5000 iterations, different $M_t$ and $r_t$]{%
\includegraphics[width=0.35\textwidth]{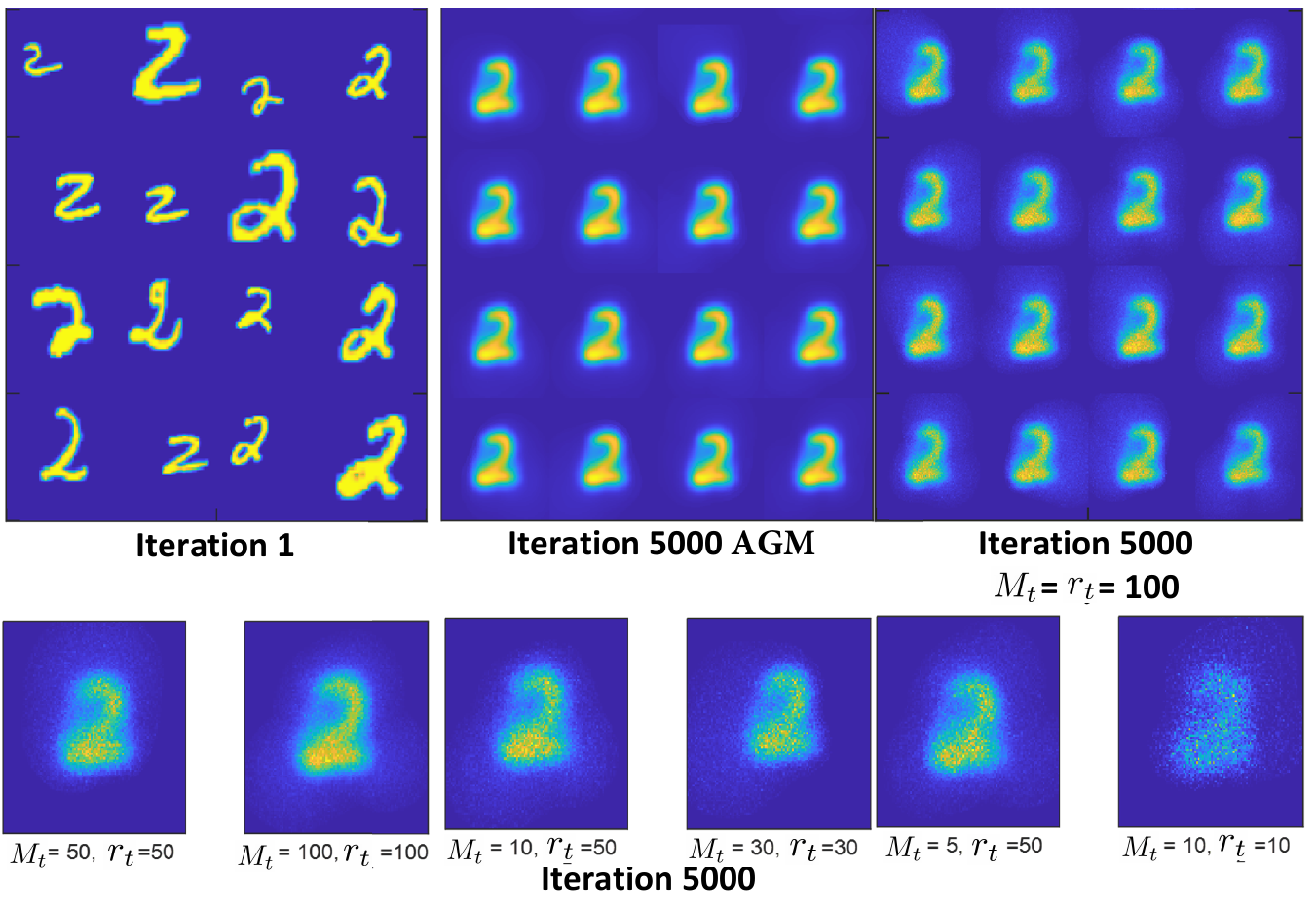}
       \label{fig:bad}
    }
    \hspace{0.01\linewidth} 
    \subfigure[Gaussians. Different number of iterations, different $M_t$ and $r_t$]{%
     \includegraphics[trim=35 20 43 20,clip,width=0.4\linewidth]{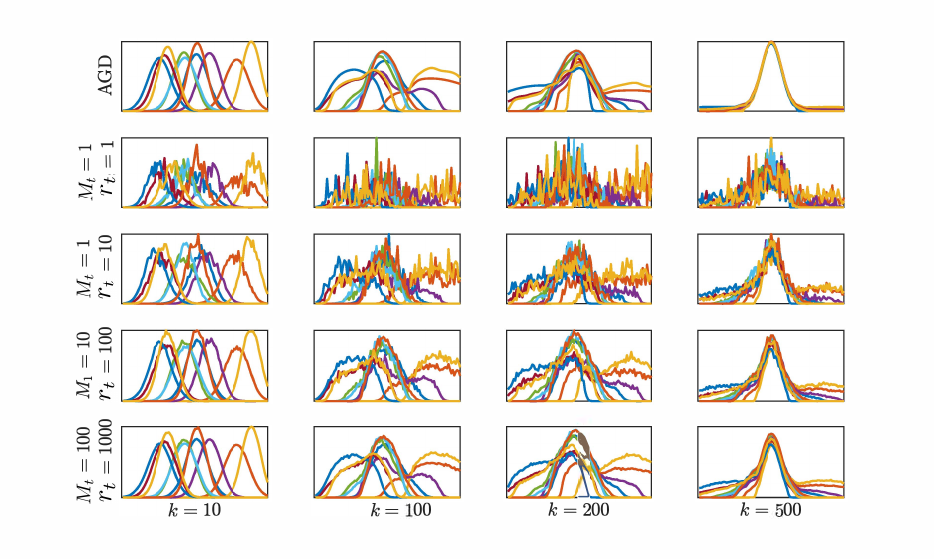}
        
        \label{fig:gaussians2vis}
    }
    \caption{Visualised approximate barycenters obtained by Algorithm~\ref{alg:devdec} and comparison with AGM}
    \label{fig:second}
\end{figure}

\subsection{Log-sum-exp}
Let $n, m \in \mathbb{N}$, $A \in \mathbb{R}^{n \times m}$, and $\boldsymbol{b} \in \mathbb{R}^n$. Let us pose the problem of optimizing the following function
\begin{equation*}
    f(\boldsymbol{x}) = \ln \left(\sum_{i=1}^n b_i \mathrm{e}^{A_i^\top \boldsymbol{x}}\right)
\end{equation*}
over $\boldsymbol{x} \in \mathbb{R}^n$. Problems of this kind arise in entropy-linear programming as a dual problem \cite{gasnikov2016efficient}, and, since the log-sum-exp function serves as a smooth approximation of $\max$ and $\|\boldsymbol{x}\|_\infty$, in some formulations of the PageRank problem or for solving systems of linear equations. $\nabla f(\boldsymbol{x})$ is $L$-Lipschitz continuous with $L = \max_{i=1,...,m} \sum_{j=1}^n A_{ji}^2$ \cite{gasnikov2016efficient}.

The gradient of $f$ is such that
\begin{equation*}
    [\nabla f(\boldsymbol{x})]_i = \frac{\sum_{j=1}^n A_{ji} b_j  \mathrm{e}^{A_j^\top \boldsymbol{x}}}{\sum_{j=1}^n b_j \mathrm{e}^{A_j^\top \boldsymbol{x}}},
\end{equation*}
for every $i = 1, ..., n$. One can see that
\begin{equation*}
    \|\nabla f(\boldsymbol{x})\|_1 = \sum_{i=1}^n [\nabla f(\boldsymbol{x})]_i = \frac{\sum_{j=1}^n \left(\sum_{i=1}^n A_{ji}\right) b_j  \mathrm{e}^{A_j^\top \boldsymbol{x}}}{\sum_{j=1}^n b_j \mathrm{e}^{A_j^\top \boldsymbol{x}}},
\end{equation*}
which implies that if $A$ is a row-stochastic matrix, i.e. $\sum_{i=1}^n A_{ji} = 1$ for all $j = 1, ..., n$, $\|\nabla f(\boldsymbol{x})\|_1 = 1$ for all $\boldsymbol{x} \in \mathbb{R}^n$. This allows one to use the simplified version of the PPS oracle \eqref{eq:pps}, which is
\begin{equation*}
    PPS_s: (x, \boldsymbol{\xi}, \boldsymbol{k}) \mapsto \frac{1}{M} \sum_{i=1}^M \boldsymbol{e}_{k_i},
\end{equation*}
for which the estimate of the variance from the Lemma~\ref{thm:var} turns into
\begin{equation*}
    \sigma_{r,M}^2 = 50 \left(\frac{n - 1}{\mathrm{e} n M} + \frac{\sigma^2}{r}\right),
\end{equation*}
and the number of bits needed to encode it reduces to $M \log_2 n$.

\section{Discussion}
In this paper, we have introduced the PPS quantization method whose result can be encoded with the nearly-optimal number of bits. The second moment of PPS is proportional to the radius of the domain set in $1$-norm, which makes this quantization especially convenient for optimizing functions with gradients taking values in $\sigma(n)$ or more general a bounded set. Moreover, we have proposed accelerated primal-dual stochastic gradient algorithms with mini-batching (with variable batch size) and quantization for problems with affine constraints. We propose to choose quantization parameters such that the inexactness introduced by it is comparable to that of the oracle after batching so that better gradient approximation requires more bits communicated at each iteration. We provide large deviation estimates of suboptimality in function value and discrepancy in constraints, which match the optimal ones up to $\widetilde{O}(1)$ factor. Based on the algorithm for affine-constrained optimisation, we have designed the distributed algorithm for decentralised optimisation. Numerical experiments with computing WB of Gaussians and images confirm the effectiveness and competitiveness of the approach. In Corollary~\ref{cor:main}, we have estimated the communication efficiency of our algorithm, relating it to the properties of the network graph, like the number of nodes, diameter, and maximum degree. Distinct terms therein differ in their dependence on problem parameters: one reflects an iteration complexity (it is neglectable if $\sigma_i$ is large, that is, the less exact is an oracle, the less information we ask from it; it is small if the network is dense), other two describe the approximation complexity (one is large when the diameter is large, another depends only on the number of nodes, but is controllable by regularization, which makes the latter playing a role in control of distributed computing).

\section*{Acknowledges}
The work of D. Pasechnyuk was supported by a grant for research centers in the field of artificial intelligence, provided by the Analytical Center for the Government of the Russian Federation in accordance with the subsidy agreement (agreement identifier 000000D730321P5Q0002) and the agreement with the Ivannikov Institute for System Programming of the Russian Academy of Sciences dated November 2, 2021 No. 70-2021-00142. The work of C.A. Uribe is partially funded by the National Science Foundation Grant \#2211815 and the Google Research Scholar Award. 

\newpage
\bibliographystyle{plainnat}
\bibliography{bibl}

\newpage
\appendix
\section{Technical statements}
\subsection{Common}
\begin{lemma}[Lemma 2 \cite{devolder2011stochastic}]
    Algorithm~\ref{alg:dev} ensures that
    \begin{align*}
        A_T \varphi(\lambda_T) &\leq \min_{\lambda \in \mathbb{R}^m} \left\{\frac{\beta_T}{2} \|\lambda\|^2 +  \sum_{t=0}^T \alpha_t (\Phi(\mu_t, \boldsymbol{\xi}_t) + G_t^\top (\lambda - \mu_t)))\right\} +\sum_{t=0}^T \alpha_t (\varphi(\mu_t) - \Phi(\mu_t, \boldsymbol{\xi}_t)) +\\
        &\quad+\sum_{t=1}^T A_{t-1} (\nabla \varphi(\mu_t) - G_t)^\top (\mu_t - \lambda_{t-1}) + \sum_{t=0}^T \frac{A_t}{\beta_t - L} \|\nabla \varphi(\mu_t) - G_t\|^2.
    \end{align*}
\end{lemma}
\begin{lemma}[Theorem 3.1.4 \cite{dvinskikh2021decentralized} (1st) and Theorem 10 \cite{devolder2011stochastic} (2nd, 3rd, and 4th)] For $\delta \in (0, 1)$, it holds that
    \begin{gather*}
        P\left(2 R_* \left\|\sum_{t=0}^T \alpha_t (\nabla \varphi(\mu_t) - G_t)\right\| \geq 2\sqrt{2} R_* \left(1 + \sqrt{3 \ln\frac{5}{\delta}}\right) \sqrt{\sum_{t=0}^T \alpha_t^2 \sigma_{r_t,M_t}^2}\right) \leq \frac{\delta}{5},\\
        P\left(\sum_{t=0}^T \alpha_t (\nabla \varphi(\mu_t) - G_t)^\top \mu_t \geq \sqrt{3 \ln \frac{5}{\delta}} (R_* + R) \sqrt{\sum_{t=0}^T \alpha_t^2 \sigma_{r_t,M_t}^2}\right) \leq \frac{\delta}{5},\\
        P\left(\sum_{t=1}^T A_{t-1} (\nabla \varphi(\mu_t) - G_t)^\top (\mu_t - \lambda_{t-1}) \geq \sqrt{3 \ln \frac{5}{\delta}} R_* \sqrt{\sum_{t=0}^T \alpha_t^2 \sigma_{r_t,M_t}^2}\right) \leq \frac{\delta}{5},\\
        P\left(\sum_{t=0}^T \frac{A_t}{\beta_t - L} \|\nabla \varphi(\mu_t) - G_t\|^2 \geq \left(1 + \ln \frac{5}{\delta}\right) \sum_{t=0}^T \frac{A_t \sigma_{r_t,M_t}^2}{\beta_t - L}\right) \leq \frac{\delta}{5},
    \end{gather*}
    if $\|\mu_t - \lambda_*\| \leq R_*, \forall t = 0, ..., T$.
\end{lemma}

\subsection{For Theorem~\ref{thm:pdconv}}
\begin{lemma}[Theorem 3.1.4 \cite{dvinskikh2021decentralized}]
    If $\|\mu_t - \lambda_*\| \leq R_*, \forall t = 0, ..., T$, $\Phi$ is defined as in Section~\ref{sec:aff}, $G_t$ is defined by Algorithm~\ref{alg:dev}, then it holds that 
    \begin{align*}
        \min_{\lambda \in \mathbb{R}^m} \Bigg\{\frac{\beta_T}{2} \|\lambda\|^2 + \sum_{t=0}^T \alpha_t (\Phi(\mu_t, \boldsymbol{\xi}_t) &+ G_t^\top (\lambda - \mu_t)))\Bigg\} + \sum_{t=0}^T \alpha_t (\varphi(\mu_t) - \Phi(\mu_t, \boldsymbol{\xi}_t)) =\\
        = \min_{\lambda \in \mathbb{R}^m} \Bigg\{\frac{\beta_T}{2} \|\lambda\|^2 + \sum_{t=0}^T \alpha_t (\varphi(\mu_t) &+ G_t^\top (\lambda - \mu_t)))\Bigg\} \leq \frac{\beta_T R^2}{2} - A_T f(\hat{x}_T) - 2 A_T R_* \|A \hat{x}_T - b\| +\\
        &+ 2 R_* \left\|\sum_{t=0}^T \alpha_t (\nabla \varphi(\mu_t) - G_t)\right\| + \sum_{t=0}^T \alpha_t (\nabla \varphi(\mu_t) - G_t)^\top \mu_t,
    \end{align*}
    where $\hat{x}_T = \frac{1}{A_T} \sum_{t=0}^T x(-A^\top \mu_t)$.
\end{lemma}

It can be shown that $R_*$ is proportional to $R$ up to factor $J(T) = O(poly(\log T))$, see Lemma G.3 and Theorem G.4 \cite{gorbunov2019optimal}. 

\begin{cor}
    For $\delta \in (0, 1)$, given the pair of primal \eqref{eq:prim} and dual \eqref{eq:dual} problems, Algorithm~\ref{alg:dev} ensures that
    \begin{equation*}
        P\left(\varphi(\lambda_T) + f(\hat{x}_T) + 2 R_* \|A\hat{x}_T - b\| \geq \frac{\beta_T R^2}{2 A_T} + \frac{C_1 R}{A_T} \sqrt{\sum_{t=0}^T \alpha_t^2 \sigma_{r_t,M_t}^2} + \frac{C_2}{A_T} \sum_{t=0}^T \frac{A_t \sigma_{r_t,M_t}^2}{\beta_t - L}\right) \leq \frac{4\delta}{5},
    \end{equation*}
    where $C_2 = 1 + \ln \frac{5}{\delta} = O(\log 1/\delta)$ and
    \begin{gather*}
        C_1 = \left(2 J(T) + \sqrt{2} - 1\right) \left(\sqrt{2} + (\sqrt{2} + 1) \sqrt{3 \ln\frac{5}{\delta}}\right) + \sqrt{2} - 2 = O(poly(\log T) \sqrt{\log 1/\delta}).
    \end{gather*}
\end{cor}

It remains to deduce from the latter statement the estimates of the form \eqref{eq:target}, yet with implicit $\varepsilon$ which is underdefined until $\boldsymbol{\alpha}$, $\boldsymbol{\beta}$, $\boldsymbol{r}$, and $\boldsymbol{M}$ are specified, which will be done in Section~\ref{sec:distr}. Note that weak duality implies $f(\hat{x}_T) - f(x_*) \leq \varphi(\lambda_T) + f(\hat{x}_T)$. The following statement allows us to move from $\hat{x}_t$ to $x_T$.

\begin{lemma}[Theorem G.4 \cite{gorbunov2019optimal}]
    For $\delta \in (0, 1)$, Algorithm~\ref{alg:dev} ensures with probability at least $1 - \delta/5$ that the following holds:
    \begin{gather*}
        f(\hat{x}_T) \geq f(x_T) - L \|\hat{x}_T - x_T\|,\\
        2 R_* \|A \hat{x}_T - b\| \geq 2 R_* \|A x_T - b\| - 2 R_* \|A\|_2 \|\hat{x}_T - x_T\|,\\
        \|\hat{x}_T - x_T\| \leq \frac{\sqrt{2}}{A_T} \frac{1 + \sqrt{3 \ln \frac{5}{\delta}}}{\|A\|} \sqrt{\sum_{t=0}^T \alpha_t^2 \sigma_{r_t,M_t}^2}.
    \end{gather*}
\end{lemma}

\subsection{For the choice of coefficients}
\begin{lemma}
    If $\alpha_t = \frac{t + 1}{a}$, $A_T = \sum_{t=0}^T \alpha_t$, then
    \begin{gather*}
        \frac{1}{A_T} \sqrt{\sum_{t=0}^T \alpha_t^2} = \sqrt{\frac{2(2 T + 3)}{3(T + 1)(T + 2)}} \leq \frac{2}{\sqrt{3}}\frac{1}{\sqrt{T}} ,\quad\frac{1}{A_T} \sum_{t=0}^T \frac{A_t}{(t+2)^{3/2}} \leq \frac{2}{3} \frac{1}{\sqrt{T}}.
    \end{gather*}
\end{lemma}

\subsection{For Corollary \ref{cor:main}}
\begin{cor}[Theorem~\ref{th:choice}]
    Given $\varepsilon > 0$, $0 < \delta < 1$, and $\boldsymbol{M}$ determined by \ref{option:1b}, Algorithm~\ref{alg:devdec} can ensure $P(f(\boldsymbol{x}) - f(\boldsymbol{x}_*) \leq \varepsilon) > 1 - \delta$ and $P\left(\frac{\|\sqrt{\boldsymbol{W}} \boldsymbol{x}\|}{\sqrt{\lambda_2(W)}} \leq \frac{\varepsilon \sqrt{m}}{B_*}\right) > 1 - \delta$, if each node $j$ sends,
    \begin{enumerate}
        \item if $r_t = r \geq 1$, \newline
        $\widetilde{O}\left(d_j \frac{B^2 \log n}{\|W\|_2 m} \max\left\{\frac{r}{\sigma_i^2} \sqrt{\frac{B_*^2}{\gamma \varepsilon} \chi(W)}, \frac{B_*^2}{\varepsilon^2} \chi(W), \frac{m^3 \lambda_2(W)}{\gamma^2 \varepsilon^2}\chi(W)\right\}\right)$ ,
        \item if \inlineequation[eq:r]{r_t = \max\left\{1, \frac{34 \gamma \sigma_i^2 \alpha_t}{\varepsilon} \chi(W) \max\left\{18 C_2, \left(C_3 + C_4 \frac{\sqrt{m^3 \lambda_2(W)}}{\gamma B_*}\right)^2\right\}\right\},\qquad\;\;\quad}\newline
        $\widetilde{O}\left(d_j\frac{B^2 \log n}{\|W\|_2 m} \max\left\{\frac{1}{\sigma_i^2} \sqrt{\frac{B_*^2}{\gamma \varepsilon} \chi(W)},\frac{B_*^2}{\varepsilon^2} \chi(W), \frac{m^3 \lambda_2(W)}{\gamma^2 \varepsilon^2}\chi(W)\right\}\right)$ bits in total.
    \end{enumerate}
\end{cor}

\end{document}